\documentclass[11pt,a4paper]{article}
\usepackage{amsmath}
\usepackage{amsfonts}
\usepackage{amssymb}
\usepackage{graphicx}

\newtheorem{thm}{Theorem}[section]
\newtheorem{lemma}[thm]{Lemma}
\newtheorem{cor}[thm]{Corollary}
\newtheorem{pro}[thm]{Proposition}
\newtheorem{de}{Definition}[section]

\newenvironment{proof} {\par \noindent \textbf{Proof: }}{\QED \par \bigskip \par}
\newcommand{\QED}{\hfill$\square$}
\newcommand{\rz}{\vspace{0.2cm}}

\begin{document}
\baselineskip 16pt

\phantom{.} \vskip 6cm

\begin{center}
{\Large \bf ON THE WIENER INDEX AND LAPLACIAN COEFFICIENTS \\[8pt]
            OF GRAPHS WITH GIVEN DIAMETER OR RADIUS \footnote{%
          This work was supported by
          the research program P1-0285 of the Slovenian Agency for Research and
          the research grant 144015G of the Serbian Ministry of Science.
}
}

\bigskip
\bigskip
{\large \sc Aleksandar Ili\'c}

\smallskip
{\em Faculty of Sciences and Mathematics, University of Ni\v s, Serbia} \\
e-mail: {\tt aleksandari@gmail.com}

\bigskip
{\large \sc Andreja Ili\' c}

\smallskip
{\em Faculty of Sciences and Mathematics, University of Ni\v s, Serbia} \\
e-mail: {\tt ilic\_andrejko@yahoo.com}

\bigskip
{\large \sc Dragan Stevanovi\'c}

\medskip
{\em University of Primorska---FAMNIT, Glagolja\v ska 8, 6000 Koper, Slovenia, and \\
     Mathematical Institute, Serbian Academy of Science and Arts, \\
     Knez Mihajlova 36, 11000 Belgrade, Serbia}\\
e-mail: {\tt dragance106@yahoo.com}

\bigskip\medskip
{\small (Received October 13, 2008)}

\end{center}


\begin{abstract}
Let $G$ be a simple undirected $n$-vertex graph with
the characteristic polynomial of its Laplacian matrix~$L(G)$,
$\det (\lambda I - L (G))=\sum_{k = 0}^n (-1)^k c_k \lambda^{n - k}$.
It is well known that for trees
the Laplacian coefficient~$c_{n-2}$ is equal to the Wiener index of~$G$.
Using a result of Zhou and Gutman on the relation between
the Laplacian coefficients and the matching numbers in
subdivided bipartite graphs,
we characterize first the trees with given diameter and
then the connected graphs with given radius
which simultaneously minimize all Laplacian coefficients.
This approach generalizes recent results of Liu and Pan
[MATCH Commun. Math. Comput. Chem. 60 (2008), 85--94]
and Wang and Guo [MATCH Commun. Math. Comput. Chem. 60 (2008), 609--622]
who characterized $n$-vertex trees with fixed diameter $d$ which minimize the Wiener index.
In conclusion,
we illustrate on examples with Wiener and modified hyper-Wiener index that
the opposite problem of simultaneously maximizing all Laplacian coefficients has no solution.
\end{abstract}

\section{Introduction}

Let $G = (V, E)$ be a simple undirected graph with $n = |V|$
vertices. The Laplacian polynomial $P (G,\lambda)$ of $G$ is the
characteristic polynomial of its Laplacian matrix $L (G) = D (G) -
A(G)$,
$$
P (G, \lambda) = \det (\lambda I_n - L (G)) = \sum_{k = 0}^n (-1)^k
c_k \lambda^{n - k}.
$$

The Laplacian matrix $L(G)$ has non-negative eigenvalues $\mu_1
\geqslant \mu_2 \geqslant \ldots \geqslant \mu_{n - 1} \geqslant
\mu_n = 0$~\cite{CvDS95}. From Viette's formulas, $c_k = \sigma_k (\mu_1, \mu_2,
\ldots, \mu_{n - 1})$ is a symmetric polynomial of order $n - 1$. In
particular, $c_0 = 1$, $c_1 = 2 n$, $c_n = 0$ and $c_{n - 1} = n
\tau (G)$, where $\tau (G)$ denotes the number of spanning trees of~$G$.
If $G$ is a tree, coefficient $c_{n - 2}$ is equal to its Wiener index,
which is a sum of distances between all pairs of vertices.\rz

Let $m_k (G)$ be the number of matchings of $G$ containing exactly
$k$ independent edges. The subdivision graph $S (G)$ of $G$ is
obtained by inserting a new vertex of degree two on each edge of~$G$.
Zhou and Gutman \cite{ZhGu08} proved that for every acyclic graph~$T$
with $n$ vertices
\begin{equation}
\label{eq-zhou} c_k (T) = m_k (S (T)), \quad 0 \leqslant k \leqslant
n.
\end{equation}

Let $C(a_1, \ldots, a_{d-1})$ be a caterpillar obtained
from a path~$P_d$ with vertices $\{v_{0},v_{1},\dots,v_d\}$
by attaching $a_i$ pendent edges to vertex~$v_{i}$, $i=1,\dots,d-1$.
Clearly, $C(a_1, \ldots, a_{d-1})$ has diameter~$d$ and
$n = d+1+\sum_{i = 1}^{d-1} a_i$.
For simplicity,
$C_{n, d} = C (0, \ldots, 0, a_{\lfloor d / 2 \rfloor}, 0, \ldots, 0)$.

\begin{figure}[ht]
  \center
  \includegraphics [width = 12cm]{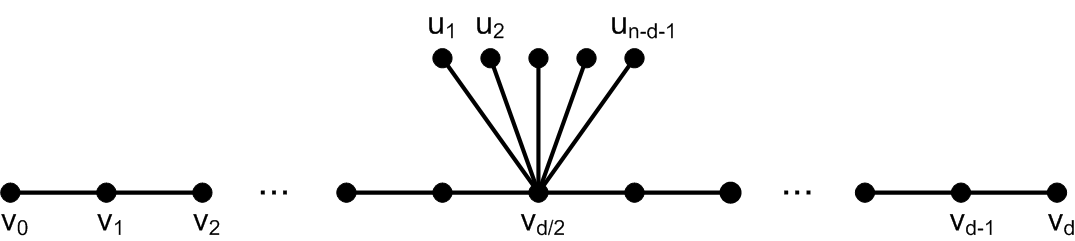}
  \caption {\textit{Caterpillar $C_{n, d}$.}}
\end{figure}

In \cite{SiMaBe08} it is shown that caterpillar $C_{n, d}$ has
minimal spectral radius (the greatest eigenvalue of adjacency
matrix) among graphs with fixed diameter. \rz

Our goal here is to characterize the trees with given diameter and
the connected graphs with given radius which simultaneously minimize
all Laplacian coefficients. We generalize recent results of Liu and Pan \cite{LiPa08}, and Wang and Guo \cite{WaGu08}
who proved that the caterpillar $C_{n, d}$ is the unique tree with $n$ vertices and diameter $d$,
that minimizes Wiener index. We also deal with connected $n$-vertex graphs with fixed diameter,
and prove that $C_{n, 2r - 1}$ is extremal graph. \rz

After a few preliminary results in Section~2, we prove in Section~3
that a caterpillar~$C_{n, d}$ minimizes all Laplacian coefficients
among $n$-vertex trees with diameter~$d$. In particular, $C_{n,d}$
minimizes the Wiener index and the modified hyper-Wiener index among
such trees. Further, in Section~4 we prove that~$C_{n, 2r-1}$
minimizes all Laplacian coefficients among connected $n$-vertex
graphs with radius~$r$. Finally, in conclusion we illustrate on
examples with Wiener and modified hyper-Wiener index that the
opposite problem of simultaneously maximizing all Laplacian
coefficients has no solution.

\section{Preliminaries}

The distance $d (u, v)$ between two vertices $u$ and $v$ in a
connected graph $G$ is the length of a shortest path between them.
The eccentricity $\varepsilon (v)$ of a vertex $v$ is the maximum
distance from $v$ to any other vertex.

\begin{de}
The diameter $d (G)$ of a graph G is the maximum eccentricity over
all vertices in a graph, and the radius $r (G)$ is the minimum
eccentricity over all $v \in V (G)$.
\end{de}

Vertices of minimum eccentricity form the center (see
\cite{DoEn01}). A tree $T$ has exactly one or two adjacent center
vertices. For a tree $T$,
\begin{equation}
\label{tree_center} d (T) = \left\{
\begin{array}{l l}
  2 r(T) - 1 & \quad \mbox{if $T$ is bicentral, }\\
  2 r (T) & \quad \mbox{if $T$ has unique center vertex. }\\
\end{array} \right.
\end{equation}

The next lemma counts the number of matchings in a path~$P_n$.

\begin{lemma}
For $0 \leqslant k \leqslant \lceil \frac{n}{2} \rceil$, the number
of matchings with $k$ edges for path $P_n$ is
$$
m_k (P_n) = \binom{n - k}{k}.
$$
\end{lemma}

\begin{proof}
If $v$ is a pendent vertex of a graph $G$, adjacent to $u$, then for
the matching number of $G$ the recurrence relation holds
$$
m_k (G) = m_k (G - v) + m_{k - 1} (G - u - v).
$$

If $G$ is a path, then $m_k (P_n) = m_k (P_{n - 1}) + m_{k - 1}
(P_{n - 2})$. For base cases $k = 0$ and $k = 1$ we have $m_0 (P_n)
= 1$ and $m_1 (P_n) = n - 1$. After substituting formula for $m_k
(P_n)$, we get the well-known identity for binomial coefficients.
$$
\binom{n - k}{k} = \binom{n - 1 - k}{k} + \binom{n - 2 - (k - 1)}{k
- 1}.
$$

Maximum cardinality of a matching in the path $P_n$ is $\lceil
\frac{n}{2} \rceil$ and thus, $0 \leqslant k \leqslant \lceil
\frac{n}{2} \rceil$.
\end{proof}

The union $G = G_1 \cup G_2$ of graphs $G_1$ and $G_2$ with disjoint
vertex sets $V_1$ and $V_2$ and edge sets $E_1$ and $E_2$ is the
graph $G = (V, E)$ with $V = V_1 \cup V_2$ and $E = E_1 \cup E_2$.
If $G$ is a union of two paths of lengths $a$ and $b$, then $G$ is
disconnected and has $a + b$ vertices and $a + b - 2$ edges.

\begin{lemma}
\label{le-two paths} Let $m_k (a, b)$ be the number of $k$-matchings
in $G = P_a \cup P_b$, where $a + b$ is fixed even number. Then, the
following inequality holds
$$
m_k \left ( \left \lceil \frac{a + b}{2} \right \rceil, \left
\lfloor \frac{a + b}{2} \right \rfloor \right ) \leqslant \ldots
\leqslant m_k (a + b - 2, 2) \leqslant m_k (a + b, 0) = m_k (P_{a +
b}).
$$
\end{lemma}

\begin{proof}
Without loss of generality, we can assume that $a \geqslant b$.
Notice that the number of vertices in every graph is equal to $a +
b$. The path $P_{a + b}$ contains as a subgraph $P_{a'} \cup
P_{b'}$, where $a' + b' = a + b$ and $a' \geqslant b' > 0$. This
means that the number of $k$-matchings of $P_{a + b}$ is greater
than or equal to the number of $k$ matchings of $P_{a'} \cup
P_{b'}$, and therefore $m_k (a + b, 0) \geqslant m_k (a', b')$. In
the sequel, we exclude $P_{a + b}$ from consideration.\rz

For the case $k = 0$, by definition we have identity $m_0 (G) = 1$.
For $k = 1$ we have equality, because
$$
m_1 (a', b') = (a' - 1) + (b' - 1) = a' + b' - 2 = a + b - 2.
$$

We will use mathematical induction on the sum $a + b$. The base
cases $a + b = 2, 4, 6$ are trivial for consideration using previous
lemma. Suppose now that $a + b$ is an even number greater than $6$
and consider graphs $G = P_a \cup P_b$ and $G' = P_{a'} \cup
P_{b'}$, such that $a > b$ and $a' = a - 2$ and $b' = b + 2$. We
divide the set of $k$-matchings of $G'$ in two disjoint subsets
$\mathcal{M}_1'$ and $\mathcal{M}_2'$. The set $\mathcal{M}_1'$
contains all $k$-matchings for which the last edge of $P_{a'}$ and
the first edge of $P_{b'}$ are not together in the matching, while
$\mathcal{M}_2'$ consists of $k$-matchings that contain both the
last edge of $P_{a'}$ and the first edge of $P_{b'}$. Analogously
for the graph $G$, we define the partition $\mathcal{M}_1 \cup
\mathcal{M}_2$ of the set of $k$-matchings. \rz

Consider an arbitrary matching $M'$ from $\mathcal{M}_1'$ with $k$
disjoint edges. We can construct corresponding matching $M$ in the
graph $G$ in the following way: join paths $P_{a'}$ and $P_{b'}$ and
form a path $P_{a + b - 1}$ by identifying the last vertex on path
$P_{a'}$ and the first vertex on path $P_{b'}$. This way we get a
$k$-matching in path $P_{a + b - 1} = v_1 v_2 \ldots v_{a + b - 1}$.
Next, split graph $P_{a + b - 1}$ in two parts to get $P_{a} = v_1
v_2 \ldots v_a$ and $P_{b} = v_a v_{a + 1} \ldots v_{a + b - 1}$.
Note that the last edge in $P_{a}$ and the first edge in $P_{b}$ are
not both in the matching $M$. This way we establish a bijection
between sets $\mathcal{M}_1$ and $\mathcal{M}_1'$. \rz

Now consider a matching $M'$ of $G'$ such that the last edge of
$P_{a'}$ and the first edge of $P_{b'}$ are in $M'$. The cardinality
of the set $\mathcal{M}'_2$ equals to $m_{k - 1} (a' - 2, b' - 2)$,
because we cannot include the first two vertices from $P_{a'}$ and
the last two vertices from $P_{b'}$ in the matching. Analogously, we
conclude that $|\mathcal{M}_2| = m_{k - 1} (a - 2, b - 2)$. This way
we reduce problem to pairs $(a' - 2, b' - 2)$ and $(a - 2, b - 2)$
with smaller sum and inequality
$$
m_{k - 1} (a - 2, b - 2) \geqslant m_{k - 1} (a' - 2, b' - 2)
$$
holds by induction hypothesis. If one of numbers in the set $\{a, b,
a', b'\}$ becomes equal to $0$ using above transformation, it must
the smallest number $b$. In that case, we have $m_{k - 1} (a - 2, 0)
\geqslant m_{k - 1} (a' - 2, b' - 2)$ which is already considered.
\end{proof}

\begin{lemma}
\label{le-radius} For every $2 \leqslant r \leqslant \lfloor
\frac{n}{2} \rfloor$, it holds
$$
c_k (C_{n, 2r}) \geqslant c_k (C_{n, 2r - 1}).
$$
\end{lemma}

\begin{proof}
Coefficients $c_0$ and $c_n$ are constant, while trees $C_{n, 2r}$
and $C_{n, 2r - 1}$ have equal number of vertices and thus, we have
equalities
$$
c_1 (C_{n, 2r}) = c_1 (C_{n, 2r - 1}) = 2n \quad \mbox{ and } \quad
c_{n - 1} (C_{n, 2r}) = c_{n - 1} (C_{n, 2r - 1}) = n.
$$

Assume that $2 \leqslant k \leqslant n - 2$. Using identity
(\ref{eq-zhou}), we will establish injection from the set of
$k$-matchings of subdivision graph $S (C_{n, 2r-1})$ to $S (C_{n,
2r})$. Let $v_0, v_1, \ldots, v_{2r-1}$ be the vertices on the main
path of caterpillar $C_{n, 2r-1}$ and $u_1, u_2, \ldots, u_{n - 2r}$
pendent vertices from central vertex $v_r$. We obtain graph $C_{n,
2r}$ by removing the edge $v_r u_{n-2r}$ and adding the edge
$v_{2r-1} u_{n-2r}$. Assume that vertices $w_1, w_2, \ldots, w_{n -
2r}$ are subdivision vertices of degree $2$ on edges $v_r u_1, v_r
u_2, \ldots, v_r u_{n - 2r}$. \rz

Consider an arbitrary matching $M$ of subdivision graph $S (C_{n,
2r-1})$. If $M$ does not contain the edge $v_r w_{n - 2r}$ then the
corresponding set of edges in $S (C_{n, 2r})$ is also a
$k$-matching. Now assume that matching $M$ contains the edge $v_r
w_{n - 2r}$. If we exclude this edge from the graph $S (C_{n, 2r -
1})$, we get graph $G' = S (C_{n, 2r-1}) - v_r w_{n - 2r} = P_{2r}
\cup P_{2r - 2} \cup (n-2r-1) P_2 \cup P_1$. Therefore, the number
of $k$-matchings that contain an edge $v_r w_{n - 2r}$ in $S (C_{n,
2r-1})$ is equal to the number of matchings with $k - 1$ edges in
graph $G'$ that is union of paths $P_{2r}$ and $P_{2(r - 1)}$ and $n
- 2r - 1$ disjoint edges $u_1 w_1, u_2 w_2, \ldots,$ $u_{n - 2r - 1}
w_{n - 2r - 1}$
$$
\mathcal{S'} = m_{k - 1} (G') = m_{k - 1} (P_{2r} \cup P_{2(r - 1)}
\cup (n - 2r - 1) P_2).
$$

On the other side, let $G$ be the graph $S (C_{n, 2r}) - v_{2r-1}
w_{n -2r}$. Since $G$ contains as a subgraph $P_{2 (2r - 1)} \cup (n
- 2r - 1) P_2$, the number of $k$-matchings that contain the edge
$v_{2r-1} w_{n - 2r}$ is greater than or equal to the number of
$(k-1)$-matchings in the union of path $P_{2 (2r - 1)}$ and $n - 2r -
1$ disjoint edges. Therefore,
$$
\mathcal{S} = m_{k - 1} (G) \geqslant m_{k - 1} (P_{2 (2r - 1)} \cup
(n - 2r - 1) P_2).
$$

Path $P_{2 (2r - 1)}$ is obtained by adding an edge that connects
the last vertex of $P_{2r}$ and the first vertex of $P_{2(r - 1)}$,
and thus we get inequality
$$
m_k (S (C_{n, 2r})) \geqslant m_k (S (C_{n, 2r-1})).
$$
Finally we get that all coefficients of Laplacian polynomial of $C_{n, 2r}$ are
greater than or equal to those of $C_{n, 2r - 1}$.
\end{proof}

\begin{figure}[ht]
  \center
  \includegraphics [width = 13.5cm]{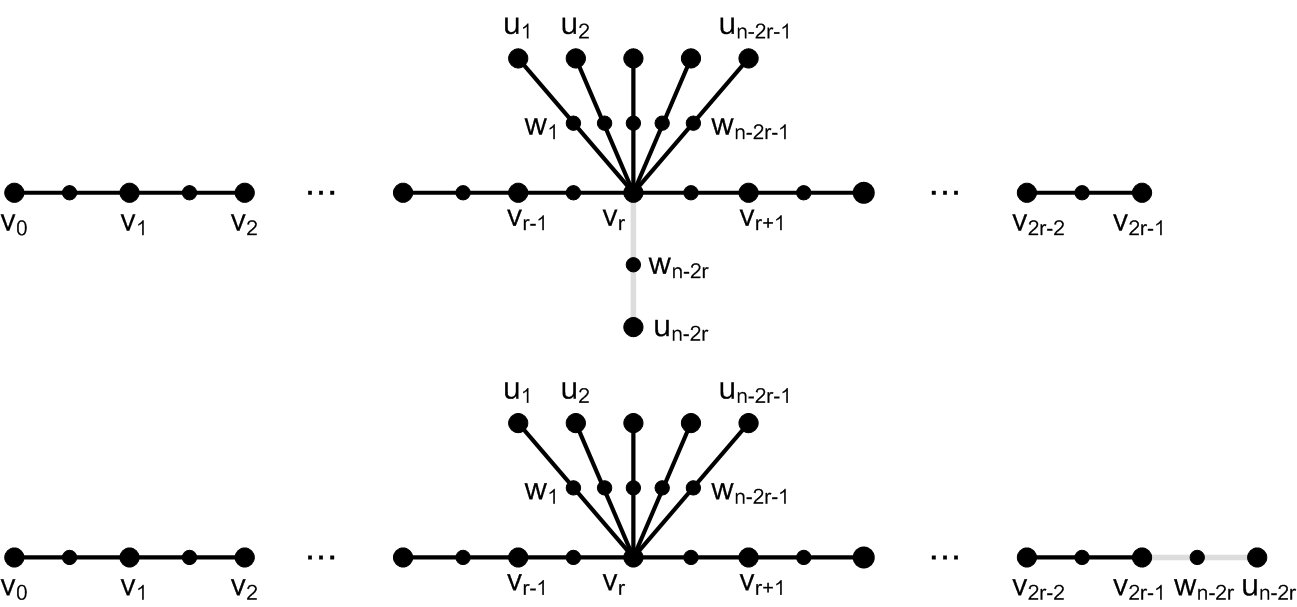}
  \caption { \textit{ Correspondence between caterpillars $C_{n, 2r - 1}$ and $C_{n, 2r}$. } }
\end{figure}

The Laplacian coefficient $c_{n - 2}$ of a tree $T$ is equal to the
sum of all distances between unordered pairs of vertices, also known
as the Wiener index,
$$
c_{n - 2} (T) = W (T) = \sum_{u, v \in V} d (u, v).
$$

The Wiener index s considered as one of the most used topological index with high correlation
with many physical and chemical indices of molecular compounds. For recent surveys
on Wiener index see \cite{DoEn01}, \cite{GuPo86}, \cite{GYLC93}.
The hyper-Wiener index $WW (G)$ \cite{Gu02} is one of the recently introduced distance based molecular
descriptors. It was proved in \cite{Gu03} that a modification of the hyper-Wiener index, denoted as $WWW (G)$,
has certain advantages over the original $WW (G)$. The modified hyper-Wiener index is
equal to the coefficient $c_{n - 3}$ of Laplacian characteristic polynomial.

\begin{pro}
The Wiener index of caterpillar $C_{n, d}$ equals:
$$
W (C_{n, d}) = \left\{
\begin{array}{l l}
\frac{d (d + 1)(d + 2)}{6} + (n - d - 1)(n - 1) + (n - d - 1) \left(\frac{d}{2} + 1 \right) \frac{d}{2}, & \quad \mbox{if $d$ is even, }\\
  \frac{d (d + 1)(d + 2)}{6} + (n - d - 1)(n - 1) + (n - d - 1)
\left(\frac{d + 1}{2} \right)^2, & \quad \mbox{if $d$ is odd. }\\
\end{array} \right.
$$
\end{pro}

\begin{proof}
By summing all distances of vertices on the main path of length $d$,
we get
$$
\sum_{i = 1}^d i (d + 1 - i) = (d + 1) \cdot \sum_{i = 1}^d i -
\sum_{i = 1}^d i^2 = \frac{d (d + 1)(d + 2)}{6}.
$$

For every pendent vertex attached to $v_{\lfloor d/2 \rfloor} = v_c$
we have the same contribution in summation for the Wiener index:
$$
(n - d - 2) + \left ( \sum_{i = 0}^d | i - c | + 1 \right ) = (n -
1) + \sum_{i = 0}^d | i - c |.
$$

Therefore, based on parity of $d$ we easily get given formula.
\end{proof}

\section{Trees with fixed diameter}

We need the following definition of $\sigma$-transformation,
suggested by Mohar in \cite{Mo07}.

\begin{de}
Let $u_0$ be a vertex of a tree $T$ of degree $p + 1$. Suppose that
$u_0 u_1, u_0 u_2, \ldots, u_0 u_p$ are pendant edges incident with
$u_0$, and that $v_0$ is the neighbor of $u_0$ distinct from $u_1,
u_2, \ldots, u_p$. Then we form a tree $T' = \sigma (T, u_0)$ by
removing the edges $u_0 u_1, u_0 u_1, \ldots, u_0 u_p$ from $T$ and
adding $p$ new pendant edges $v_0 v_1, v_0 v_2, \ldots, v_0 v_p$
incident with $v_0$. We say that $T'$ is a $\sigma$-transform of
$T$.
\end{de}

Mohar proved that every tree can be transformed into a star by a sequence of $\sigma$-transformations.

\begin{thm} [\cite{Mo07}]
Let $T' = \sigma(T, u_0)$ be a $\sigma$-transform of a tree $T$ of
order $n$. For $d = 2, 3, \ldots k$, let $n_d$ be the number of
vertices in $T - u_0$ that are at distance $d$ from $u_0$ in $T$.
Then
$$
c_k (T) \geqslant c_k (T') + \sum_{d = 2}^k n_d \cdot p \cdot \binom{n - 2
- d}{k - d} \quad \mbox { for } 2 \leqslant k \leqslant n - 2
$$
and $c_k (T) = c_k (T')$ for $k \in \{0, 1, n - 1, n\}$.
\end{thm}

\begin{thm}
\label{thm-diameter} Among connected acyclic graphs on $n$ vertices
and diameter $d$, caterpillar
$$
C_{n, d} = C (0, \ldots, 0, a_{\lfloor d / 2 \rfloor}, 0, \ldots,
0),
$$
where $a_{\lfloor d / 2 \rfloor} = n - d - 1$, has minimal Laplacian
coefficient $c_k$, for every $k = 0, 1, \ldots, n$.
\end{thm}

\begin{proof}
Coefficients $c_0$, $c_1$, $c_{n - 1}$ and $c_n$ are constant for
all trees on $n$ vertices. The star graph $S_n$ is the unique tree
with diameter $2$ and path $P_n$ is unique graph with diameter $n -
1$. Therefore, we can assume that $2 \leqslant k \leqslant n - 2$
and $3 \leqslant d \leqslant n - 2$. \rz

Let $P = v_0 v_1 v_2 \ldots v_d$ be a path in tree $T$ of maximal
length. Every vertex $v_i$ on the path $P$ is a root of a tree $T_i$
with $a_i + 1$ vertices, that does not contain other vertices of
$P$. We apply $\sigma$-transformation on trees $T_1, T_2, \ldots,
T_{d - 1}$ to decrease coefficients $c_k$, as long as we do not get
a caterpillar $C (a_0, a_1, a_2, \ldots, a_d)$. By a theorem of Zhou
and Gutman, it suffices to see that
$$
m_k (S (C (a_1, a_2, \ldots, a_{d - 1}))) > m_k (S (C_{n, d})).
$$

Assume that $v_{\lfloor d / 2 \rfloor} = v_c$ is a central vertex of
$C_{n, d}$. Let $u_1, u_2, \ldots, u_{n - d - 1}$ be pendent
vertices attached to $v_c$ in $S (C_{n, d})$, and let $w_1, w_2,
\ldots, w_{n - d - 1}$ be subdivision vertices on pendent edges $v_c
u_1, v_c u_2, \ldots, v_c u_{n - d - 1}$. We also introduce ordering
of pendent vertices. Namely, in the graph $C_n (a_1, a_2, \ldots,
a_{d - 1})$ first $a_1$ vertices in the set $\{ u_1, u_2, \ldots, u_{n - d
- 1} \}$ are attached to $v_1$, next $a_2$ vertices are attached to
$v_2$, and so on. \rz

Consider an arbitrary matching $M'$ with $k$ edges in caterpillar $S
(C_{n, d})$. If $M$ does not contain any of the edges $ \{v_c w_1,
v_c w_2, \ldots, v_c w_{n - d - 1} \}$, then we can a construct
matching in $S (C (a_1, a_2, \ldots, a_{d - 1}))$, by taking
corresponding edges from $M$. If the edge $v_c w_i$ is in the
matching $M'$ for some $1 \leqslant i \leqslant n - d - 1$, the
corresponding edge $v_j w_i$ is attached to some vertex $v_j$, where
$1 \leqslant j \leqslant d - 1$. Moreover, if we fix the number $l$
of matching edges in the set
$$
\{ u_1 w_1, u_2 w_2, \ldots, u_{i - 1} w_{i - 1}, u_{i + 1} w_{i +
1}, \ldots, u_{n - d - 1} w_{n - d - 1} \},
$$
we have to choose exactly $k - l - 1$ independent edges in the
remaining graphs. Caterpillar $S (C_{n, d})$ is decomposed into two
path of lengths $2 \lfloor \frac{d}{2} \rfloor$ and $2 \lceil
\frac{d}{2} \rceil$, and caterpillar $S (C (a_1, a_2, \ldots,
a_{d - 1}))$ is decomposed in paths of lengths $2j$ and $2d - 2j$. From
Lemma \ref{le-two paths} we can see that
$$
m_{k - l - 1} (2 \left \lfloor \frac{d}{2} \right \rfloor, 2 \left
\lceil \frac{d}{2} \right \rceil) \leqslant m_{k - l - 1} (2j, 2d -
2j).
$$

If we sum this inequality for $l = 0, 1, \ldots, k - 1$, we obtain
that the number of $k$-matchings in graph $S (C_{n, d})$ is less
than the number of $k$-matchings in $S (C (a_0, a_1, a_2, \ldots,
a_d))$. Thus, for every tree $T$ on $n$ vertices with diameter $d$
holds:
$$
c_k (C_{n, d}) \leqslant c_k (T), \quad k = 0, 1, 2, \ldots n.
$$
\end{proof}

\begin{figure}[ht]
  \center
  \includegraphics [width = 10.5cm]{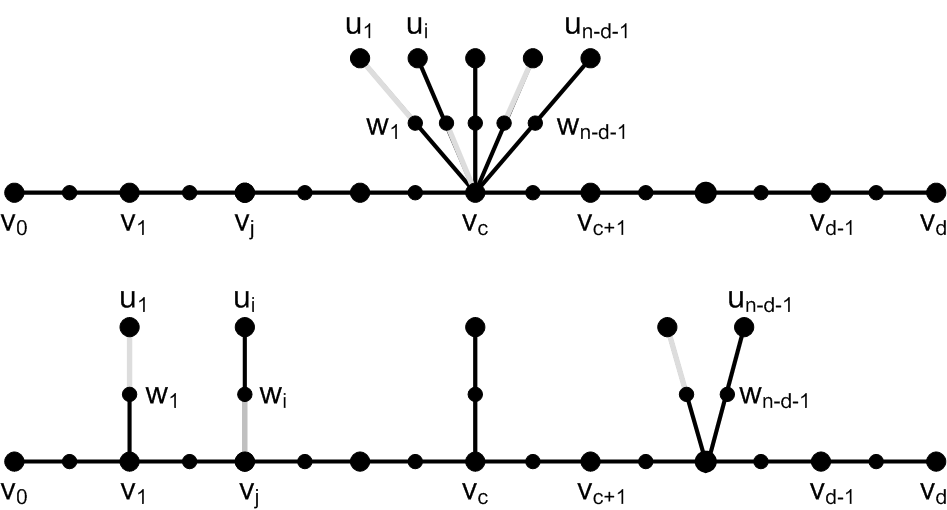}
  \caption { \textit{ Correspodence between caterpillars $C_{n, d}$ and $C (a_0, a_1, \ldots, a_d)$.  } }
\end{figure}

\section{Graphs with fixed radius}

\begin{thm}
Among connected graphs on $n$ vertices and radius $r$, caterpillar
$C_{n, 2r-1}$ has minimal coefficient $c_k$, for every $k = 0, 1,
\ldots, n$.
\end{thm}

\begin{proof}
Let $v$ be a center vertex of $G$ and let $T$ be a spanning tree of $G$ with shortest paths
from $v$ to all other vertices.
Tree $T$ has radius $r$ and can be obtained by performing the breadth first search algorithm (see \cite{CoLRS01}).
Laplacian eigenvalues of an edge-deleted graph $G - e$ interlace those of $G$,
$$
\mu_1 (G) \geqslant \mu_1 (G - e) \geqslant \mu_2 (G) \geqslant
\mu_2 (G - e) \geqslant \ldots \geqslant \mu_{n - 1} (G) \geqslant
\mu_{n - 1} (G - e) \geqslant 0.
$$

Since, $c_k (G)$ is equal to $k$-th symmetric polynomial of eigenvalues
$(\mu_1 (G), \mu_2 (G), \ldots, \mu_{n - 1} (G))$, we have $c_k (G) \geqslant c_k (G -
e)$. Thus, we delete edges of $G$ until we get a tree $T$ with radius
$r$. This way we do not increase Laplacian coefficients $c_k$. The diameter of tree $T$
is either $2r-1$ or $2r$. Since $c_k (C_{n, 2r-1}) \leqslant c_k (C_{n, 2r})$ from Lemma \ref{le-radius}
we conclude that extremal graph on $n$ vertices, which has minimal
coefficients $c_k$ for fixed radius $r$, is the caterpillar $C_{n,
2r-1}$.
\end{proof}

We can establish analogous result on the Wiener index.

\begin{cor}
Among connected graphs on $n$ vertices and radius $r$, caterpillar
$C_{n, 2r-1}$ has minimal Wiener index.
\end{cor}

\section{Concluding remarks}

We proved that $C_{n, 2r - 1}$ is the unique graph that minimize all
Laplacian coefficients simultaneously among graphs on $n$ vertices
with given radius $r$. In the class of $n$-vertex graphs with fixed
diameter, we found the graph with minimal Laplacian coefficients in
case of trees---because it is not always possible to find a spanning
tree of a graph with the same diameter. \rz

Naturally, one wants to describe $n$-vertex graphs with fixed radius
or diameter with maximal Laplacian coefficients. We have checked all trees up to $20$ vertices and
classify them based on diameter and radius. For every triple $(n, d, k)$
and $(n, r, k)$ we found extremal graphs with $n$ vertices and fixed
diameter $d$ or fixed radius $r$ that maximize coefficient $c_k$.
The result is obvious---trees that maximize Wiener index are different
from those with the same parameters that maximize modified hyper-Wiener index. \rz

The following two graphs on the Figure 4 are extremal for $n = 18$
vertices with diameter $d = 4$; the first graph is a unique tree
that maximizes Wiener index $c_{n - 2} = 454$ and the second one is
also a unique tree that maximizes modified hyper-Wiener index $c_{n
- 3} = 4960$. \rz

\begin{figure}[ht]
  \label{fig-dia}
  \center
  \includegraphics [width = 4cm]{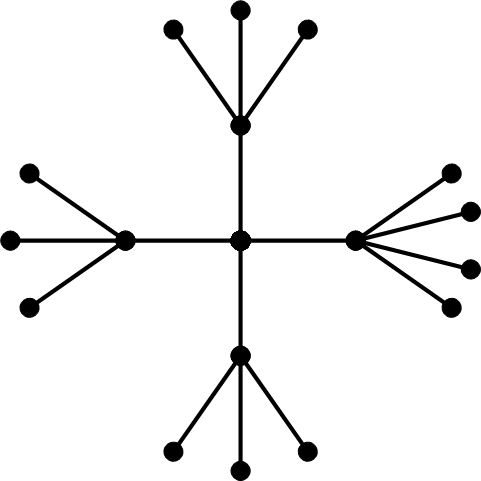}
  \hspace{2cm}
  \includegraphics [width = 4cm]{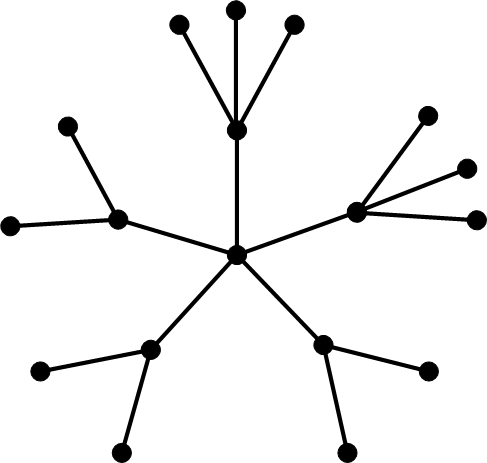}
  \caption { \textit{ Graphs with $n = 18$ and $d = 4$ that maximize $c_{16}$ and $c_{15}$. } }
\end{figure}

The following two graphs on the Figure 5 are extremal for $n = 17$
vertices with radius $r = 5$; the first graph is a unique tree that
maximizes Wiener index $c_{n - 2} = 664$ and the second one is also
a unique tree that maximizes modified hyper-Wiener index $c_{n - 3}
= 9173$.

\begin{figure}[ht]
  \label{fig-rad}
  \center
  \includegraphics [width = 8.5cm]{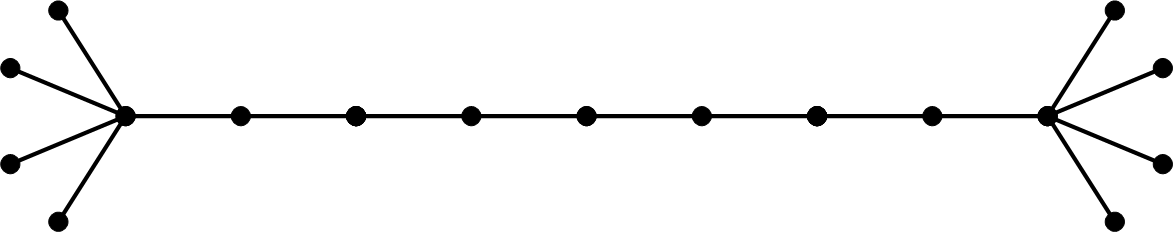}
    \hspace{1cm}
  \includegraphics [width = 7cm]{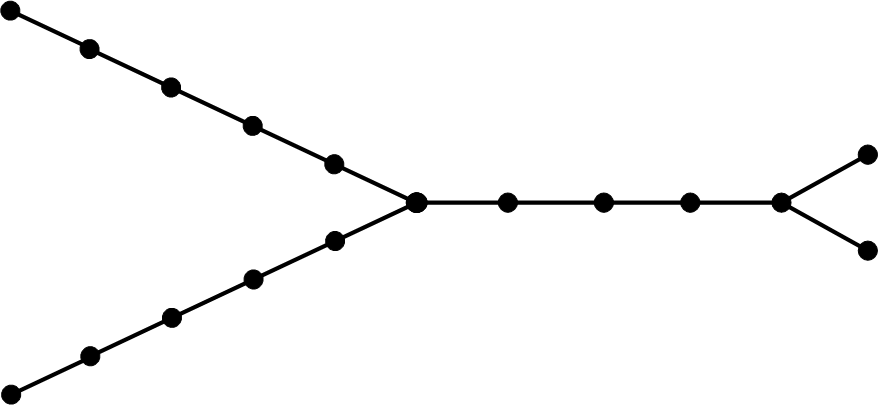}
  \caption { \textit{ Graphs with $n = 17$ and $r = 5$ that maximize $c_{15}$ and $c_{14}$. } }
\end{figure}


\begin{thebibliography}{99}

\bibitem{BeWa07}
    S. Bereg, H. Wang, \textit{Wiener indices of balanced binary trees}, Discr. Appl. Math.
    155 (2007), 457--467.

\bibitem{CvDS95}
    D. Cvetkovi\' c, M. Doob, H. Sachs, \textit{Spectra of graphs - Theory and Application}, 3rd edition,
    Johann Ambrosius Barth Verlag, 1995.

\bibitem{CoLRS01}
    T. H. Cormen, C. E. Leiserson, R. L. Rivest, C. Stein,
    \textit{Introduction to Algorithms}, Second Edition, MIT Press, Cambridge, MA, 2001.

\bibitem{DoEn01}
    A. Dobrynin, R. Entringer, I. Gutman, \textit{Wiener index of
    trees: theory and applications}, Acta Appl. Math. 66 (2001), 211–-249.

\bibitem{GuPo86}
    I. Gutman, O. E. Polansky, \textit{Mathematical concepts in organic chemistry}, Springer-Verlag,
    Berlin 1986.

\bibitem{GYLC93}
    I. Gutman, Y. N. Yeh, S. L. Lee, J. C. Chen, \textit{Some recent results in the theory of the Wiener number},
    Indian J. Chem. 32A (1993), 651--661.

\bibitem{Gu02}
    I. Gutman, \textit{Relation between hyper-Wiener and Wiener index}, Chem. Phys. Lett.
    364 (2002), 352--356.

\bibitem{Gu03}
    I. Gutman, \textit{Hyper-Wiener index and Laplacian spectrum}, J. Serb. Chem. Soc. 68 (2003),
    949--952.

\bibitem{GuPa03}
    I. Gutman, L. Pavlovi\' c, \textit{On the coefficients of the Laplacian characteristic
    polynomial of trees}, Bull. Acad. Serbe Sci. Arts 127 (2003) 31--40.

\bibitem{LiPa08} H.~Liu, X.F.~Pan,
    \textit{On the Wiener index of trees with fixed diameter}, MATCH
    Commun. Math. Comput. Chem. 60 (2008), 85--94.

\bibitem{Mo07}
    B. Mohar, \textit{On the Laplacian coefficients of acyclic graphs},
    Linear Algebra Appl. 722 (2007), 736--741.

\bibitem{SiMaBe08}
    S. K. Simi\' c, E. M. L. Marzi, F. Belardo,
    \textit{On the index of caterpillars}, Discrete Math. 308 (2008), 324--330.

\bibitem{St08}
    D. Stevanovi\' c, \textit{Laplacian-like energy of trees},
    MATCH Commun. Math. Comput. Chem. 61 (2009), 407--417.

\bibitem{YaYe06}
    W. Yan, Y. N. Yeh, \textit{Connections between Wiener index and
    matchings}, J. Math. Chem. 39 (2006), 389--399.

\bibitem{WaGu08}
    S. Wang, X. Guo, \textit{Trees with extremal Wiener indices},
    MATCH Commun. Math. Comput. Chem. 60 (2008), 609--622.

\bibitem{ZhLi99}
    F. Zhang, H. Li,
    \textit{On acyclic conjugated molecules with minimal energies}, Discr. Appl. Math. 92 (1999), 71--84.

\bibitem{ZhGu08} B. Zhou, I. Gutman, \textit{A connection between ordinary and Laplacian spectra of
    bipartite graphs}, Linear Multilin. Algebra 56 (2008), 305--310.


\end{thebibliography}
\end{document}